\documentclass[12pt]{article}
\usepackage{hyperref,graphicx,fancyhdr,fancybox,rotating,xcolor,appendix}
\usepackage{ulem,hhline,dsfont,pifont,multicol,lmodern}
\usepackage{amscd}
\usepackage{lipsum}
\usepackage[all,cmtip]{xy}
\usepackage{amsmath,amsthm,amsfonts,amsbsy,amssymb,geometry,times,pst-node}
\usepackage{mathrsfs} 
\usepackage{latexsym}
\usepackage{mathtools}
\usepackage{enumitem} 
\usepackage[utf8]{inputenc}
\usepackage[english]{babel}
\usepackage{ mathrsfs }

\usepackage{graphicx}

\usepackage[all]{xy}
\def\Mod{{\rm Mod}}

\def\Hom{{\rm Hom}}
\def\Im{{\rm Im}}
\def\Ker{{\rm Ker}}

\def\id{{\rm id}}

\def\A{{\mathscr{A}}}
\def\C{{\mathscr{C}}}
\def\K{{\mathscr{K}}}

\newtheorem{thm}{\bf Theorem}[section]
\newtheorem{cor}[thm]{\bf Corollary}
\newtheorem{lem}[thm]{\bf Lemma}
\newtheorem{prop}[thm]{\bf Proposition}
\newtheorem{Def}[thm]{\bf Definition}

\newtheorem{ex}[thm]{\bf Example}

\begin{document}
	\title{A new approach to projectivity in the categories of complexes,  II}
	\author{Driss Bennis, Juan Ram\'on Garc\'{\i}a Rozas,\\
	 Hanane Ouberka and Luis Oyonarte}

	\date{}
	
	\maketitle

\bigskip

\noindent{\large\bf Abstract.}  It is now very known how the subprojectivity of modules provides a fruitful new unified framework of the classical projectivity and flatness. In this paper, we extend this fact to the category of complexes by generalizing and unifying several known classical results.  We further provide various examples to illustrate the scopes and limits of the established results.\\
\indent This paper is a continuation of a recent work in which it was shown among other several things that the subprojectivity of complexes can be characterized in terms of morphisms in the homotopy category.


\bigskip

\small{\noindent{\bf Key Words.} Subprojectivity domain of a complex; Subprojectivity domain of a class complexes; projective complex; flat complex}

\small{\noindent{\bf 2010 Mathematics Subject Classification.} 16E05}




\section{Introduction}
Throughout the paper $R$ will denote an associative ring with unit (nonnecessarily commutative). The category of left $R$-modules will be denoted by $R-\Mod$. Modules are, unless otherwise explicitly stated, left $R$-modules.  
The category of complexes of $R$-modules will be denoted by $\C(R)$. For two complexes $X$ and $Y$, we use $\Hom_{\C(R)}(X,Y)$ to present the group of all morphisms of complexes from $X$ to $Y$.  The homotopy category will be denoted by $\K(R)$.  For two complexes $X$ and $Y$, we use $\Hom_{\K(R)}(X,Y)$ to present the homotopy equivalence classes of morphisms of complexes from $X$ to $Y$.\\  
 Given a class of modules $\mathcal{L}$, a complex $X: \xymatrix{\cdots \ar[r] & X_{i+1}\ar[r]^{d_{i+1}} & X_{i}\ar[r]^{d_i}  &X_{i-1}\ar[r] &\cdots}$ is said to be $\Hom_R(\mathcal{L},-)$-exact (resp., $\Hom_R(-,\mathcal{L})$-exact) if the complex of abelian groups $\Hom_R(L,N)$ (resp., $\Hom_R(N,L)$) is exact for every $L \in \mathcal{L}$. We will  denote by $\epsilon_n^X: X_{n} \to \Im d_{n}$ the canonical epimorphism and by $\mu_n^X : \Ker(d_{n-1}) \to X_{n-1}$ the canonical monomorphism. The $n^{th}$ boundary (resp., cycle, homology) of a complex $X$ is defined as $\Im\, d_{n+1}^X$ (resp., $\Ker\, d_n^X$,  $ \Ker\, d_n / \Im\, d_{n+1}$) and it is denoted by $B_n(X)$ (resp., $Z_n(X)$, $H_n(X)$).  If there is a bound $b$ such that $X_n=0$ for all  $n\geqslant b$ (resp. $n \leqslant b$) we say that the complex $X$ is bounded above (resp. bounded below). When a complex is bounded above and below,  we say simply say it is bounded.\\
Also, throughout the paper,  for a module $M$, we denote by $\overline{M}$ the complex  with all terms $0$ except $M$ in the degrees $1$ and $0$ with identity morphism of $M$. Also, we denote by $\underline{M}$ the complex  with all terms $0$ except $M$ in the degree $0$. Given a complex $X$ with differential $d^X$ and an integer $n$, we denote by $X[n]$ the complex consisting of $X_{i-n}$ in degree $i$ with differential $(-1)^nd_{i-n}^X$.\\

 Given two objects $M$ and $N$ of an abelian category $\A$ with enough projectives, $M$ is said to be $N$-subprojective if for every epimorphism $g:B\to N$ and every morphism $f:M\to  N$, there exists a morphism $h:M\to B$ such that $gh=f$, or equivalently, if every morphism $M\to N$ factors through a projective object (see \cite[Proposition 2.7]{SubprojAb}). The subprojectivity domain of any object $M$, denoted by ${\underline {\mathfrak{Pr}}}_{\A}^{-1}(M)$, is defined as the class of all objects $N$ such that $M$ is $N$-subprojective, and the subprojectivity domain of a whole class $\mathfrak{C}$ of $\A$, denoted by ${\underline {\mathfrak{Pr}}}_{\A}^{-1}(\mathfrak{C})$, is defined as the class of objects $N$ such that every $C$ of $\mathfrak{C}$ is $N$-subprojective.  In \cite[Theorem 3.11]{Subcom}, it is proved that,  for any two complexes $M$ and $N$ with $N_{n+1} \in {\underline {\mathfrak{Pr}}}_{R-\Mod}^{-1}(M_n)$ for every $n\in \mathbb{Z}$, the conditions $N \in {\underline {\mathfrak{Pr}}}_{\C(R)}^{-1}(M)$ and $\Hom_{\K(R)}(M,N)=0$ are equivalent. As a consequence of \cite[Theorem 3.11]{Subcom},       subprojectivity of some classes were determined. Namely, it is shown that the subprojectivity of the  class $\{\underline{R}[n]\ n \in \mathbb{Z}\}$ is the class of exact complexes (see  \cite[Corollary 3.17]{Subcom}).\\

The notion of subprojectivity was introduced in \cite{Sergio} as a new treatment in the analysis of the projectivity of a module. However, the study of the subprojectivity goes beyond that goal and, indeed, provides, among other things, a new and interesting perspective on some other known notions. In this way, an alternative perspective on the projectivity of an object of an abelian category $\A$ with enough projectives was investigated in \cite{SubprojAb}. It is proved among several things, that  the subprojectivity provides  a new   unified  framework of the classical  projectivity and flatness. Namely,   one can see clearly that, for    an abelian category $\A$  with enough projectives,   ${{\underline {\mathfrak{Pr}}}_{\A}}^{-1}(\mathcal{P}_{\A})=\mathcal{\A}$, where $\mathcal{P}_{\A}$ denotes the class of projective objects of $\mathcal{A}$.  On the other hand, we have  ${{\underline {\mathfrak{Pr}}}_{\A}}^{-1}(\A)=\mathcal{P}_{\A}$ and ${{\underline {\mathfrak{Pr}}}_{\A}}^{-1}(\mathcal{FP_{\A}})=\mathcal{F}_{\A}$, where $\mathcal{FP}_{\A}$ denotes the class of finitely presented objects, and $\mathcal{F}_{\A}$ denotes the class of flat objects (an object $F$ is said to be flat if every short exact sequence $0 \to A \to B \to F \to 0$ is pure, that is, if for every finitely presented object $P$, $\Hom_{\A}(P, -)$ makes this sequence exact, \cite{Sten}).

In \cite{Subcom},  a first   treatment of   the    subprojectivity in the category of complexes is done. It is  proved among several things  that the concept of subprojectivity  in the category of complexes is   closely  linked to that of null-homotopy of morphisms.  In this paper, we go further in the study of subprojectivity in   the category of complexes. 
Namely, we approve that the  subprojectivity in the category of complexes provides also a new interesting unified  framework of the classical  projectivity and flatness from which arise several natural questions that we will answer throughout the paper. The structure of the paper is as follows:
 
 In Section 2, we study the relationship between the subprojectivity of a complex and the subprojectivity of its cycles. As a first natural question, inspired by some classical facts,  we ask whether for an exact complex $M$ and a complex $N$, $ N \in {{\underline {\mathfrak{Pr}}}_{\C(R)}}^{-1}(M)$ if $N_{n} \in {{\underline {\mathfrak{Pr}}}_{R-\Mod}}^{-1}(Z_{n-1}(M))$ for every $n \in \mathbb{Z}$. In Proposition \ref{prop-QF}, we show that  this holds for every exact complex $M$ and every bounded complex $N$ only on quasi-Fr\"obenius rings (rings over which every projective module is injective or equivalently, every injective module is projective). Then, we are interested in the subprojectivity domain of   contractible complexes, that is, complexes of the form $\oplus_{n \in \mathbb{Z}} \overline{M_n}[n]$ for some family of modules $\{M_n\}_{n \in \mathbb{Z}}$.   For   a complex $N$  and  a family of modules $\{M_n\}_{n \in \mathbb{Z}}$, we show that  $N \in {{\underline {\mathfrak{Pr}}}_{\C(R)}}^{-1}(\oplus_{n \in \mathbb{Z}}\overline{M_n}[n])$ if and only if $N_{n+1} \in  {{\underline {\mathfrak{Pr}}}_{R-\Mod}}^{-1}(M_n)$ for every $n \in \mathbb{Z}$. 
This can be seen as a new extension of the known fact that the  projective complexes are exactly the contractible ones  with projective cycles.  As a   second main result of Section  2,  we prove that for an exact complex $N$    and  a bounded below complex $M$, $N \in  {{\underline {\mathfrak{Pr}}}_{\C(R)}}^{-1}(M)$ if every $Z_n(N) \in  {{\underline {\mathfrak{Pr}}}_{R-\Mod}}^{-1}(M_n)$ 
(see Theorem  \ref{thm-bound}). Recall that a complex is finitely presented if and only if it is bounded and its terms are all finitely presented modules (see for instance  \cite[Lemma 4.1.1]{JR}). So the second main result is a natural generalization of the characterization of flat complexes being as complexes which are exact with flat cycles.  
At the end of Section 2,  relations between the conditions given in Theorem  \ref{thm-bound} are discussed. Namely, we show that the condition ``$M$  is bounded below'' cannot be dropped  (see Example \ref{exm-bounded}). Also, we give an example showing that the reverse implication of Theorem \ref{thm-bound} does not hold true in general (see Example \ref{exm-inv-main1}).

 Section $3$ is devoted to the study of  subprojectivity domains of classes of complexes.  Namely, inspired by some classical facts, we focus our study on the  classes of complexes  constructed from a class of modules  $\mathcal{L}$:

\begin{itemize} 
	\item The class  of complexes $X$ such that every $X_n \in \mathcal{L}$ will be denoted by $\#\mathcal{L}$.
	\item The class of bounded complexes (resp., bounded below complexes) $X$ such that every $X_n \in \mathcal{L}$ will be denoted by $\C^b(\mathcal{L})$ (resp.,  $\C^-(\mathcal{L})$).
	\item  The class of exact complexes $X$ such that every $Z(X)_n \in \mathcal{L}$ will be denoted by $\mathcal{\widetilde{L}}$.
\end{itemize}

For some particular cases of classes of modules $\mathcal{L}$, the classes of complexes $\#\mathcal{L}$,  $\C^b(\mathcal{L})$ and  $\mathcal{\widetilde{L}}$ are usual. For instance, we have  $\C(R)=\#{R-\Mod}$,    $\mathcal{ P}_{\C(R)}=\widetilde{\mathcal{P_{R-\Mod}}}$ (i.e., the class of projective complexes)   and   $ \mathcal{FP}_{\C(R)}=\C^b(\mathcal{FP}_{R-\Mod})$  (i.e., the class of finitely presented complexes). Thus,  the following questions arise naturally: Let  $\mathcal{L}$ and $\mathcal{G}$ be two classes of modules  such that ${{\underline {\mathfrak{Pr}}}_{R-\Mod}}^{-1}(\mathcal{L})=\mathcal{G}$:
\begin{enumerate}
\item  When do we have ${{\underline {\mathfrak{Pr}}}_{\C(R)}}^{-1}(\mathcal{\widetilde{L}})=\#\mathcal{G}$? 
\item  When do we have $ {{\underline {\mathfrak{Pr}}}_{\C(R)}}^{-1}(\#\mathcal{L})=\mathcal{\widetilde{G}}$?
\item   When do we have $ {{\underline {\mathfrak{Pr}}}_{\C(R)}}^{-1}(\C^b(\mathcal{L}))=\mathcal{\widetilde{G}}$?
\end{enumerate}

 	We will show in Theorem \ref{thm-5-1}, that if $0 \in \mathcal{L} \cap \mathcal{G}$, then ${{\underline {\mathfrak{Pr}}}_{\C(R)}}^{-1}(\mathcal{\widetilde{L}})=\#\mathcal{G}$ if and only if ${{\underline {\mathfrak{Pr}}}_{R-\Mod}}^{-1}(\mathcal{L})=\mathcal{G}$ and  $\Hom_{\K(R)}(M, N)=0$ for any $M \in \mathcal{\widetilde{L}}$ and $N \in \#\mathcal{G}$.
  \\
  
   To answer the second question, we will show in Theorem \ref{thm-5-2}, that if $0, R \in \mathcal{L}$. Then, $ {{\underline {\mathfrak{Pr}}}_{\C(R)}}^{-1}(\#\mathcal{L})=\mathcal{\widetilde{G}}$ if and only if ${{\underline {\mathfrak{Pr}}}_{R-\Mod}}^{-1}(\mathcal{L})=\mathcal{G}$ and  $\Hom_{\K(R)}(M, N)=0$ for any $M \in \#\mathcal{L}$ and $N \in \mathcal{\widetilde{G}}$. 
\\

 We end the paper with   Theorem \ref{thm-5-3} which answers Question 3 as follows:  If $0, R \in \mathcal{L}$, then ${{\underline {\mathfrak{Pr}}}_{R-\Mod}}^{-1}(\mathcal{L})=\mathcal{G}$ if and only if $ {{\underline {\mathfrak{Pr}}}_{\C(R)}}^{-1}(\C^b(\mathcal{L}))=\mathcal{\widetilde{G}}$.
\\

Some consequences of the main results are established (see Proposition \ref{pro-GP} and Corollaries \ref{cor-l1},  \ref{cor-l2} and  \ref{cor-l3}).



\section{Subprojectivity domains of particular complexes}
The purpose of this section is to study the relationship between the subprojectivity of complexes and the subprojectivity of their cycles.\\

The prove of the first main result is based on     the following lemma.

\begin{lem}\label{lem-nul3}
	Let $\mathcal{L}$ be a class of modules, $M$ be  an exact complex which is $\Hom_R(-, \mathcal{L})$-exact and $N$ a bounded complex above. If every morphism $Z_{n-1}(M) \to N_n$ factors through a module in $\mathcal{L}$, then every morphism of complexes $M \to N$ is null-homotopic by a morphism $s$ such that every $s_n$ factors through a module of $\mathcal{L}$. 
\end{lem}
\begin{proof}
	Let $f: M \to N$ be a morphism of complexes.  We are going to construct a family of morphisms $s_n:M_{n-1} \to N_{n}$, such that $f_n=d^{N}_{n+1}s_{n+1}+s_{n}d_n^{M}$. We suppose that $N_n=0$ for every $n>0$. Then $f_0 d_1^M=0$, so there exists a morphism $t_0:  Z_{-1}(M) \to N_0$ such that $t_0 \epsilon^M_0=f_0$. Then,  there exist two morphisms $\beta_0 : Z_{-1}(M) \to L_0$ and $\alpha_0 : L_0 \to N_0$ with $L_0 \in \mathcal{L}$  and $t_0=  \alpha_0 \beta_0$. Since $M$ is $\Hom_R(-, \mathcal{L})$-exact, there exists a morphism $\gamma_0:M_{-1}\to L_0$ such that $\gamma _0 \mu_0^M =\beta_0$.
	$$\xymatrix{
		\cdots \ar[rr]&&M_0\ar[rrr]^{d_0^M}\ar[rd]^{\epsilon_0^M}\ar[rddd]_{f_0}&&&M_{-1}\ar[rr]^{d_{-1}^M}\ar[ldd]_{\gamma_0}\ar[rddd]_{f_{-1}}&& \cdots\\
		&&&Z_{-1}(M)\ar[rd]_{\beta_0}\ar[rru]^{\mu_0^M}\ar[dd]_{t_0}&&&&\\
		&&&&L_0\ar[ld]_{\alpha_0}&&&\\
		0 \ar[rrr]&&&N_0\ar[rrr]_{d_{0}^N}&&&N_{-1}\ar[r]_{d_{-1}^M}& \cdots
	}$$
	Let $s_0=\alpha_0 \gamma_0$. One can check that $s_0 d_0^M=f_0$, hence $f_{-1}d_0^M=d_0^Nf_0= d_0^Ns_0 d_0^M$, so there exists a morphism $t_{-1}:  Z_{-2}(M) \to N_{-1}$ such that $t_{-1} \epsilon^M_{-1}=f_{-1}-d_0^Ns_0 $. Then,  there exist two morphisms $\beta_{-1} : Z_{-2}(M) \to L_{-1}$ and $\alpha_{-1} : L_{-1} \to N_{-1}$ with $L_{-1}\in \mathcal{L}$ and $t_{-1}=  \alpha_{-1} \beta_{-1}$.  Since $M$ is $\Hom_R(-, \mathcal{L})$-exact there exists a morphism $\gamma_{-1}:M_{-2}\to L_{-1}$ such that $\gamma _{-1} \mu_{-1}^M =\beta_{-1}$. Let $s_{-1}=\alpha_{-1}\gamma_{-1}$. Then,  $f_{-1}-d_0^Ns_0=t_{-1} \epsilon^M_{-1}=\alpha_{-1} \beta_{-1} \epsilon^M_{-1}=\alpha_{-1} \gamma _{-1} \mu_{-1}^M \epsilon^M_{-1}=\alpha_{-1} \gamma _{-1} d_{-1}^M=s_1  d_{-1}^M$, hence $f_{-1}= d_0^Ns_0+s_{-1}  d_{-1}^M$. Using the same arguments we construct, and for any $n\leqslant0$, $s_n:M_{n-1} \to N_{n}$, such that $f_n=d^{N}_{n+1}s_{n+1}+s_{n}d_n^{M}$, for $n>0$, we take $s_n=0$. Therefore, $f: M \to N$ is null homotopic by the morphism $s$ such that every $s_n$ factors through a module $L_n$ of $\mathcal{L}$.
\end{proof}

\begin{prop}\label{prop-QF}
	The following conditions are equivalent.
	\begin{enumerate}
		\item For every exact complex $M$ and every bounded complex above $N$, if for every $n \in \mathbb{Z}$, $N_{n} \in {{\underline {\mathfrak{Pr}}}_{R-\Mod}}^{-1}(Z_{n-1}(M))$, then $ N \in {{\underline {\mathfrak{Pr}}}_{\C(R)}}^{-1}(M)$.
		\item For every exact complex $M$ and every module $N$, if there exists $n \in \mathbb{Z}$ such that $N \in {{\underline {\mathfrak{Pr}}}_{R-\Mod}}^{-1}(Z_{n-1}(M))$, then  $\underline{N}[n] \in {{\underline {\mathfrak{Pr}}}_{\C(R)}}^{-1}(M)$.
		\item $R$ is quasi-Fr\"obenius.
	\end{enumerate}
\end{prop}
\begin{proof}
	$1. \Rightarrow 2.$ is clear.\\
	$2. \Rightarrow 3.$ Let $P$ be a projective module and $i:P \to E$ be a monomorphism with $E$ is injective. Let us prove that $\Hom_R(i,M):\Hom_R(E,M) \to \Hom_R(P,M)$ is epic for every module $M$. For let $f:P \to M$ be a morphism of modules and consider the exact complex $X: \, \cdots \to 0 \to P \to E \to C \to 0 \to \cdots$ with $P$ is in the  $0$ position. Then,  $\underline{M}$ holds in the subprojectivity domain of $X$ by assumption. Thus, there exist two morphisms of complexes $\beta: X \to Q$ and $\alpha:Q\to \underline{M}$ such that $Q$ is projective and $\alpha\beta=\phi$ where $\phi_0=f$ and $\phi_i=0$ otherwise. We have $\alpha_0 d_1^Q=0$, hence there exists a morphism $h:Z_{-1}(Q) \to M$ such that $h \epsilon_0^Q=\alpha_0$. Since $Q$ is projective, the morphism $\mu_0^Q:Z_{-1}(Q) \to Q_{-1}$ splits, that is, there exists a morphism $\nu_0^Q: Q_{-1} \to Z_{-1}(Q)$ such that $\nu_0^Q \mu_0^Q=\id$. Then, $$h \nu_0^Q \beta_{-1} i= h \nu_0^Q d_0^Q \beta_0= h \nu_0^Q \mu_0^Q \epsilon_0^Q \beta_0= h \epsilon_0^Q \beta_0= \alpha_0 \beta_0= \phi_0=f.$$ Thus, $\Hom_R(i,M):\Hom_R(E,M) \to \Hom_R(P,M)$ is epic for every module $M$. Therefore, $P$ is injective.\\
	$3. \Rightarrow1.$ Let $M$ be an exact complex and $N$ a bounded complex above such that, for every $n \in \mathbb{Z}$, $N_{n} \in {{\underline {\mathfrak{Pr}}}_{R-\Mod}}^{-1}(Z_{n-1}(M))$. Hence, every morphism $Z_{n-1}(M) \to N_n$ factors through a projective module. Then, by Lemma \ref{lem-nul3}, every morphism of complexes $M \to N$ is null-homotopic by a morphism $s$ such that every $s_n:M_n \to N_{n+1}$ factors through a projective module. Then, by \cite[Lemma 3.10]{Subcom}, every morphism of complexes $M \to N$ factors through a projective complex. Therefore, $ N \in {{\underline {\mathfrak{Pr}}}_{\C(R)}}^{-1}(M)$.
\end{proof}

Now, we are interested in the study of  the subprojectivity domain of   contractible complexes. It is  inspired by  the fact that projective complexes are contractibles.  

The following result treats a particular case which will be useful.

\begin{lem} \label{lem-disc}
	Let $n \in \mathbb{Z}$, $N$ be a complex and $M$ a module. Then,  $N \in  {{\underline {\mathfrak{Pr}}}_{\C(R)}}^{-1}(\overline{M}[n])$ if and only if $N_{n+1} \in  {{\underline {\mathfrak{Pr}}}_{R-\Mod}}^{-1}(M)$.
\end{lem}
\begin{proof}
	Suppose that $N \in  {{\underline {\mathfrak{Pr}}}_{\C(R)}}^{-1}(\overline{M}[n])$ and let $f:M \to N_{n+1}$ be a morphism of modules.  Then, by assumption,  there exists a morphism of complexes $\overline{f}: \overline{M}[n] \to N$ such that $\overline{f}_{n+1}=f$. $\overline{f}: \overline{M}[n] \to N$ factors through a projective complex $P$. Hence, $f:M \to N_{n+1}$ factors through the projective module $P_{n+1}$.\\
	Conversely, Suppose that $N_{n+1} \in  {{\underline {\mathfrak{Pr}}}_{R-\Mod}}^{-1}(M)$ and let $f: \overline{M}[n] \to N$ be a morphism of complexes. Then, by assumption, there exists two morphisms $\alpha:P \to N_{n+1}$ and $\beta : M\to P$ such that $f_{n+1}=\alpha\beta$. We define two morphisms of complexes $g: \overline{P}[n] \to N$ and $h: \overline{M}[n] \to \overline{P}[n]$ such that $g_{n+1}=\alpha$,  $g_{n}=d_{n+1}^N\alpha$,  $h_{n+1}=h_n=\beta$, and $g_m=h_m=0$ otherwise. It is clear that $f=gh$. Therefore, $N \in  {{\underline {\mathfrak{Pr}}}_{\C(R)}}^{-1}(\overline{M}[n])$. 
\end{proof}

\begin{prop}\label{shift-disc}
Let $N$ be a complex and $\{M_n\}_{n \in \mathbb{Z}}$ be a family of modules. Then, we get that $N \in {{\underline {\mathfrak{Pr}}}_{\C(R)}}^{-1}(\oplus_{n \in \mathbb{Z}}\overline{M_n}[n])$ if and only if $N_{n+1} \in  {{\underline {\mathfrak{Pr}}}_{R-\Mod}}^{-1}(M_n)$ for every $n \in \mathbb{Z}$.	
\end{prop}
\begin{proof}
	We have ${{\underline {\mathfrak{Pr}}}_{\C(R)}}^{-1}(\oplus_{n \in \mathbb{Z}}\overline{M_n}[n]) = {{\underline {\mathfrak{Pr}}}_{\C(R)}}^{-1}(\{\overline{M_n}[n]\ n \in \mathbb{Z}\})$, by $\cite[Proposition 2.16]{SubprojAb}$, and we conclude by Lemma \ref{lem-disc}.
\end{proof}

Now, we turn to the second aim of this section which consists of studying  the subprojectivity counterpart of the flateness. Namely, for a bounded complex $M$ and an exact complex $N$, we ask  whether $ N \in {{\underline {\mathfrak{Pr}}}_{\C(R)}}^{-1}(M)$ if $Z_{n}(N) \in {{\underline {\mathfrak{Pr}}}_{R-\Mod}}^{-1}(M_n)$ for every $n \in \mathbb{Z}$. Next, we will show that   this holds for every bounded complex below $M$. 

\begin{lem}\label{lem-nul2}
	Let $\mathcal{L}$ be a class of modules, $M$ a bounded complex below and $N$ an exact complex which is $\Hom_R(\mathcal{L}, -)$-exact. If every morphism $M_n \to Z_n(N)$ factors through a module in $\mathcal{L}$ then every morphism of complexes $M \to N$ is null-homotopic by a morphism $s$ such that every $s_n:M_n \to N_{n+1}$ factors through a module of $\mathcal{L}$. 
\end{lem}

\begin{proof}
	Let $f: M \to N$ be a morphism of complexes. We are going to construct a family of morphisms $s_n:M_n \to N_{n+1}$, such that $f_n=d^{N}_{n+1}s_n+s_{n-1}d_n^{M}$. We suppose that $M_n=0$ for every $n<0$, then $d_0^N f_0=0$, so there exists a morphism $t_0: M_0 \to Z_0(N)$ such that $\mu_1^Nt_0=f_0$. By assumption, there exist two morphisms $\beta_0 : M_0 \to L_0$ and $\alpha_0 : L_0 \to Z_{0}(N)$ with $L_0 \in \mathcal{L}$ and $t_0=  \alpha_0 \beta_0$. Since $N$ is $\Hom_R(\mathcal{L}, -)$-exact, there exists a morphism $\gamma_0:L_0\to N_1$ such that $\epsilon_1^N \gamma _0=\alpha_0$.
	$$
	\xymatrix{
		\cdots\ar[rr]^{d_{2}^M}&&M_{1}\ar[rr]^{d_{1}^M}\ar[dd]^{f_{1}}&&M_{0}\ar[rr]\ar[dd]^{f_0} \ar[ld]_{\beta_0}\ar[lddd]^{t_0}&& 0 \ar[dd]\ar[rr]&&\cdots\\
		&&&L_0\ar[ld]_{\gamma_0}\ar[dd]_{\alpha_0}&&&&&\\
		\cdots\ar[rr]^{d_{1}^N}&&N_{1}\ar[rr]^{d_{1}^N}\ar[rd]_{\epsilon_{1}^N}&&N_{0}\ar[rr]^{d_{0}^N}&& N_{-1}\ar[rr]&&\cdots\\
		&&&Z_0(N)\ar[rd]\ar[ru]_{\mu^N_1}&&&&&\\
		&&0\ar[ru]&&0&&&& \\
	}
	$$
	Let $s_0=\gamma_0 \beta_0$. One can check that $d_1^Ns_0=f_0$, hence $d_1^Ns_0d_1^M=f_0d_1^M= d_1^Nf_1$. Thus, there exists a morphism $t_1: M_1 \to Z_1(N)$ such that $\mu_2^Nt_1=f_1-s_0d_1^M$. By assumption, there exist two morphisms $\beta_1 : M_1 \to L_1$ and $\alpha_1 : L_1 \to Z_{1}(N)$ with $L_1 \in \mathcal{L}$ and $t_1=  \alpha_1 \beta_1$. Since $N$ is $\Hom_R(\mathcal{L}, -)$-exact, there exists a morphism $\gamma_1:L_1\to N_2$ such that $\epsilon_2^N \gamma _1=\alpha_1$. Let $s_1=\gamma_1 \beta_1$, then $d_2^Ns_1= \mu_2^N \epsilon_2^N \gamma_1 \beta_1=\mu_2^N \alpha_1 \beta_1=\mu_2^Nt_1=f_1-s_0d_1^M$. Using the same arguments we construct $s_n:M_n \to N_{n+1}$, such that $f_n=d^{N}_{n+1}s_n+s_{n-1}d_n^{M}$, for any $n\geqslant0$. For $n<0$, we take $s_n=0$.
\end{proof}

Now, we are in position to prove the   desired result.

\begin{thm}\label{thm-bound}
	Let $N$ be an exact complex and $M$ a bounded below complex. Then, if every $Z_n(N) \in  {{\underline {\mathfrak{Pr}}}_{R-\Mod}}^{-1}(M_n)$, then $N \in  {{\underline {\mathfrak{Pr}}}_{\C(R)}}^{-1}(M)$.
\end{thm}
\begin{proof}  Suppose that  $Z_n(N) \in  {{\underline {\mathfrak{Pr}}}_{R-\Mod}}^{-1}(M_n)$ for every $n \in \mathbb{Z}$. Then, every morphism $M_n \to Z_n(N)$ factors through a projective module. Thus, by Lemma \ref{lem-nul2}, every morphism $f:M \to N$ is null-homotopic by a morphism $s$ such that each $s_n$ factors through a projective module. Then, every $f:M \to N$ factors through a projective complex, by \cite[Lemma 3.10]{Subcom}. Thus, $N \in  {{\underline {\mathfrak{Pr}}}_{\C(R)}}^{-1}(M)$.	
\end{proof}
Recall that the subprojectivity domain of the class of all finitely presented modules is the class of all flat  modules (see \cite[Proposition 2.18]{SubprojAb}). In particular, the subprojectivity domain of a finitely presented module contains the class of all flat modules. So,  by Theorem \ref{thm-bound}, we get the following consequences.

\begin{cor}
	The subprojectivity domain of a bounded complex below of finitely presented modules contains the class of all flat complexes.
\end{cor}
\begin{cor}
	The subprojectivity domain of a bounded complex below of projective modules contains the class of all exact complexes.
\end{cor} 
The following example shows that Theorem \ref{thm-bound} fails without assuming the condition ``$M$  is bounded below''.

\begin{ex}\label{exm-bounded}
	If $R$ is   a   quasi-Fr\"obenius ring which is not  semisimple, then there exist a non bounded complex $P$ and an exact complex $E$ such that $E$ does not hold in ${{\underline {\mathfrak{Pr}}}_{\C(R)}}^{-1}(P)$ and $Z_n(E) \in  {{\underline {\mathfrak{Pr}}}_{R-\Mod}}^{-1}(P_n)$ for every $n \in \mathbb{Z}$.
\end{ex}
\begin{proof}
	Let $M$ be a non projective module and $N$ be a module such that $N$ does not hold in ${{\underline {\mathfrak{Pr}}}_{R-\Mod}}^{-1}(M)$ ($N$ exists since $M$ is not projective). Let $E$ be an exact complex with $E_{1}=N$, $E_{n}=0$ for every $n>1$ and $E_n$ is injective for every $n<1$. Let $P$ be an other exact complex with projective components and $Z_0(P)=M$ (we can construct such a complex since $R$ is quasi-Fr\"obenius). Since the components of $P$ are projectives, it is clear that for every $n \in \mathbb{Z}$, $Z_n(E) \in  {{\underline {\mathfrak{Pr}}}_{R-\Mod}}^{-1}(P_n)$. Now, suppose that $E \in  {{\underline {\mathfrak{Pr}}}_{\C(R)}}^{-1}(P)$ and let $f:M \to N$ be a morphism of modules. We construct a morphism of complexes $g:P \to E$ as follows
	$$
	\xymatrix{ 
		\cdots \ar[r]& P_{2} \ar[rr]^{d^P_{2}}\ar[dd]_{g_2=0}&& P_{1} \ar[rr]^{d^P_{1}}\ar[dd]_{g_1=f{\epsilon_1^P}}\ar[rd]_{\epsilon_1^P}&& P_{0} \ar[rr]^{d^P_{0}}\ar[dd]_{g_0}\ar[rd]_{\epsilon_0^P}&& P_{-1} \ar[r]\ar[dd]_{g_{-1}}& \cdots\\
		&&&&M\ar[ru]_{\mu_1^P}\ar[ld]^f&&Z_1(P)\ar[ru]_{\mu_0^P}\ar[dd]_{z_1}&&\\
		\cdots \ar[r]& 0 \ar[rr]&& N \ar[rr]_{d^E_{1}}&& E_0 \ar[rr]_{d^E_{0}}\ar[rd]_{\epsilon_0^E}&& E_{-1} \ar[r]& \cdots\\
		&&&&&&Z_1(E)\ar[ru]_{\mu_0^E}&&
	}
	$$
	where $g_0: P_0 \to E_0$ exists such that $g_0 \mu_1^P=d_1^E f$ since $E_0$ is injective (one can verifies that $g_0 d_1^P=d^E_1g_1$), $z_1: Z_1(P) \to Z_1(E)$ exists such that $z_1 \epsilon_0^P=\epsilon_0^E g_0$ by the universal property of cokernels, $g_{-1}: P_{-1} \to E_{-1}$ exists such that $g_{-1} \mu_0^P=\mu_0^E z_1$ since $E_{-1}$ is injective. It is clear that $g_{-1} d_0^P=d^E_0g_0$. Using the same arguments we construct $g_n: P_n \to E_n$ such that $g_{n} d_{n+1}^P=d^E_{n+1}g_{n+1}$ for every $n < 0$. For every $n>1$ we take $g_n=0$. Now, the morphism of complexes $g:P \to E$ factors through a projective complex $Q$ since we supposed that $E \in  {{\underline {\mathfrak{Pr}}}_{\C(R)}}^{-1}(P)$. Let $\beta : P \to Q$ and  $\alpha: Q \to E$ be two morphisms of complexes such that $f=\alpha \beta$. Consider the following commutative diagram
	$$
	\xymatrix{
		0 \ar[r]& M  \ar[r]^{\mu_1^P}\ar[d]_\gamma& P_0  \ar[r]^{d_0^P}\ar[d]_{\beta_0}& P_{-1}  \ar[r]\ar[d]_{\beta_{-1}}& \cdots\\
		0 \ar[r]& Z_0(Q)  \ar[r]^{\mu_1^Q}\ar[d]_\delta & Q_0  \ar[r]^{d_0^Q}\ar[d]_{\alpha_0}& Q_{-1}  \ar[r]\ar[d]_{\alpha_{-1}}& \cdots\\
		0 \ar[r]& N  \ar[r]^{d_1^E}& E_0  \ar[r]^{d_0^E}& E_{-1}  \ar[r]& \cdots\\
	}
	$$
	The morphisms $\gamma: M \to Z_0(Q)$ and $\delta : Z_0(Q) \to N$ exist and make the diagram commute, by the universal property of kernels. We claim that $f= \delta \gamma$. Indeed, $d_1^E f \epsilon_1^P=d_1^E g_1= g_0 d_1^P= \alpha_0 \beta_0 \mu_1^P \epsilon_1^P=  \alpha_0  \mu_1^Q \gamma \epsilon_1^P= d_1^E \delta \gamma \epsilon_1^P$, then $f= \delta \gamma$. Thus, any morphism $f:M \to N$ factors through a projective module. Then,  $N \in {{\underline {\mathfrak{Pr}}}_{R-\Mod}}^{-1}(M)$ which is not the case. Therefore, $E$ does not hold in ${{\underline {\mathfrak{Pr}}}_{\C(R)}}^{-1}(P)$ even if $Z_n(E) \in  {{\underline {\mathfrak{Pr}}}_{R-\Mod}}^{-1}(P_n)$ for every $n \in \mathbb{Z}$.
\end{proof}

The following example shows that the reverse implication of Theorem \ref{thm-bound} does not hold true in general.

\begin{ex}\label{exm-inv-main1}
	Let $R$ be an IF ring which is not von Neumann regular, then there exists a finitely presented module which is not projective, so there exists a module $K$ which does not belong to the subprojectivity domain of $M$. Consider an exact complex $N: \cdots \to 0 \to K \to E \to C \to 0 \to \cdots$ such that $K$ is in position $2$ and $E$ is an injective module. Since $M$ is finitely presented (then every flat module holds in its subprojectivity domain) and $R$ is an $IF$-ring,  $E \in  {{\underline {\mathfrak{Pr}}}_{R-\Mod}}^{-1}(M)$. Then, $N \in  {{\underline {\mathfrak{Pr}}}_{\C(R)}}^{-1}(\overline{M})$ (as we will see in  Lemma \ref{lem-disc}). However,  $Z_1(N)=K$ does not belong to ${{\underline {\mathfrak{Pr}}}_{R-\Mod}}^{-1}(M)$.
\end{ex}

\section{Subprojectivity domains of classes of complexes}\label{Sec 5}

In this section, we deal with the subprojectivity domain of classes of complexes.  We will answer questions raised at the end of the introduction. The first main result, which  answers the first question,  is given as follows:
 
\begin{thm}\label{thm-5-1}
	Let $\mathcal{L}$ and $\mathcal{G}$ be two classes of modules such that $\mathcal{L}$ is closed under extensions and $0 \in \mathcal{L} \cap \mathcal{G}$. Then,  the following conditions are equivalent.
	\begin{enumerate}
		\item $ {{\underline {\mathfrak{Pr}}}_{R-\Mod}}^{-1}(\mathcal{L})=\mathcal{G}$ and  $\Hom_{\K(R)}(M, N)=0$ for any $N \in \#\mathcal{G}$ and $M \in \mathcal{\widetilde{L}}$.
		\item  ${{\underline {\mathfrak{Pr}}}_{\C(R)}}^{-1}(\mathcal{\widetilde{L}})=\#\mathcal{G}$.
	\end{enumerate}
\end{thm}
\begin{proof} 
	$1. \Rightarrow 2.$ Let $N \in {{\underline {\mathfrak{Pr}}}_{\C(R)}}^{-1}(\mathcal{\widetilde{L}})$ and $M \in \mathcal{L}$. Then, $\oplus_{n\in \mathbb{Z}}\overline{M}[n] \in \mathcal{\widetilde{L}}$. Hence,  $ N \in {{\underline {\mathfrak{Pr}}}_{\C(R)}}^{-1}(\oplus_{n\in \mathbb{Z}}\overline{M}[n])$. Then, by Proposition \ref{shift-disc}, for every $n \in \mathbb{Z}$, $N_n \in {{\underline {\mathfrak{Pr}}}_{R-\Mod}}^{-1}(M)$. Thus, for every $n \in \mathbb{Z}$, $N_n \in {{\underline {\mathfrak{Pr}}}_{R-\Mod}}^{-1}(\mathcal{L})=\mathcal{G}$. Then,   ${{\underline {\mathfrak{Pr}}}_{\C(R)}}^{-1}(\mathcal{\widetilde{L}}) \subseteq \#\mathcal{G}$. Conversely, let $N \in \#\mathcal{G}$ and $M \in \mathcal{\widetilde{L}}$. Then,  for every $n \in \mathbb{Z}$, $N_n \in \mathcal{G}$ and $M_n \in \mathcal{L}$ since $\mathcal{L}$ is closed under extensions. Then,  for every $n \in \mathbb{Z}$, $N_{n+1} \in {{\underline {\mathfrak{Pr}}}_{R-\Mod}}^{-1}(M_n)$. Then,  $N \in {{\underline {\mathfrak{Pr}}}_{\C(R)}}^{-1}(M)$,  by \cite[Theorem 3.11]{Subcom}, since  $\Hom_{\K(R)}(M, N)=0$. Then,  $N \in {{\underline {\mathfrak{Pr}}}_{\C(R)}}^{-1}(\mathcal{\widetilde{L}})$. Therefore,  ${{\underline {\mathfrak{Pr}}}_{\C(R)}}^{-1}(\mathcal{\widetilde{L}})=\#\mathcal{G}$.\\
	$2. \Rightarrow 1.$ let $N \in \#\mathcal{G}$ and $M \in \mathcal{\widetilde{L}}$, then  $N \in {{\underline {\mathfrak{Pr}}}_{\C(R)}}^{-1}(M)$, hence every morphism $M \to N$ factors through a projective complex, then every morphism $M \to N$ is null homotopic (see  \cite[Corollary 3.5]{Gillespie2}). Thus,  $\Hom_{\K(R)}(M, N)=0$. Now, let us prove that $ {{\underline {\mathfrak{Pr}}}_{R-\Mod}}^{-1}(\mathcal{L})=\mathcal{G}$. For let $N \in {{\underline {\mathfrak{Pr}}}_{R-\Mod}}^{-1}(\mathcal{L})$ and $M \in \mathcal{\widetilde{L}}$, then $N \in {{\underline {\mathfrak{Pr}}}_{R-\Mod}}^{-1}(M_0)$ since every $M_0 \in \mathcal{L}$ ($\mathcal{L}$ is closed under extensions). Then,  $\overline{N} \in {{\underline {\mathfrak{Pr}}}_{\C(R)}}^{-1}(M)$ by  \cite[Lemma 3.7]{Subcom}. Then,  $\overline{N} \in {{\underline {\mathfrak{Pr}}}_{\C(R)}}^{-1}(\mathcal{\widetilde{L}})=\#\mathcal{G}$. Thus, $N \in \mathcal{G}$. Therefore,  $ {{\underline {\mathfrak{Pr}}}_{R-\Mod}}^{-1}(\mathcal{L})\subseteq \mathcal{G}$. Conversely, let $N \in \mathcal{G}$ and $M \in \mathcal{L}$, hence $\overline{N} \in \#\mathcal{G}$ and $\overline{G} \in \mathcal{\widetilde{L}}$. Then,  $\overline{N}  \in {{\underline {\mathfrak{Pr}}}_{\C(R)}}^{-1}(\overline{M})$ by assumption. Then,  by  \cite[Lemma 3.7]{Subcom}, $N \in {{\underline {\mathfrak{Pr}}}_{R-\Mod}}^{-1}(M)$, then $N \in {{\underline {\mathfrak{Pr}}}_{R-\Mod}}^{-1}(\mathcal{L})$. Therefore, $ {{\underline {\mathfrak{Pr}}}_{R-\Mod}}^{-1}(\mathcal{L})=\mathcal{G}$.
\end{proof}

Now, the second main result of this section is as follows : 

\begin{thm}
	\label{thm-5-2}
	Let $\mathcal{L}$ and $\mathcal{G}$ be two classes of modules such that $0,R \in \mathcal{L}$. Then,  the following conditions are equivalent.
	\begin{enumerate}
		\item $ \mathcal{G}={{\underline {\mathfrak{Pr}}}_{R-\Mod}}^{-1}(\mathcal{L})$ and  $\Hom_{\K(R)}(M, N)=0$ for any $M \in \#\mathcal{L}$ and $N \in \mathcal{\widetilde{G}}$.
		\item  ${{\underline {\mathfrak{Pr}}}_{\C(R)}}^{-1}(\#\mathcal{L})= \mathcal{\widetilde{G}}$ .
	\end{enumerate}
\end{thm}
\begin{proof}
	$1. \Rightarrow 2.$ Let  $N \in {{\underline {\mathfrak{Pr}}}_{\C(R)}}^{-1}(\#\mathcal{L})$, then  $N \in {{\underline {\mathfrak{Pr}}}_{\C(R)}}^{-1}(\underline{R}[n])$ for every $n \in \mathbb{Z}$, so $N$ is exact by \cite[Corollary 3.17]{Subcom}. Now, let $L \in \mathcal{L}$, then  $N \in {{\underline {\mathfrak{Pr}}}_{\C(R)}}^{-1}(\underline{L}[n])$ for every $n \in \mathbb{Z}$. Then, by \cite[Lemma 3.6]{Subcom}, $ Z_n(N) \in {{\underline {\mathfrak{Pr}}}_{R-\Mod}}^{-1}(L)$ for every  $L \in \mathcal{L}$ and $n\in \mathbb{Z}$. Then,  $ Z_n(N) \in {{\underline {\mathfrak{Pr}}}_{R-\Mod}}^{-1}(\mathcal{L})= \mathcal{G}$ for every $n \in \mathbb{Z}$. Thus, $ N \in \mathcal{\widetilde{G}}$. Conversely, let $N \in \mathcal{\widetilde{G}}$ and $M \in \#\mathcal{L}$. For every $n \in \mathbb{Z}$, $Z_n(N) \in \mathcal{G}= {{\underline {\mathfrak{Pr}}}_{R-\Mod}}^{-1}(\mathcal{L})$, hence $N_n \in {{\underline {\mathfrak{Pr}}}_{R-\Mod}}^{-1}(\mathcal{L})$ since ${{\underline {\mathfrak{Pr}}}_{R-\Mod}}^{-1}(\mathcal{L})$ is closed under extensions. Then,  $N \in {{\underline {\mathfrak{Pr}}}_{\C(R)}}^{-1}(L)$ by \cite[Theorem 3.11]{Subcom} since $\Hom_{\K(R)}(M,N)=0$ by assumption. Thus, $N \in {{\underline {\mathfrak{Pr}}}_{\C(R)}}^{-1}(\#\mathcal{L})$.\\	
	$2. \Rightarrow 1.$ Let $N \in {{\underline {\mathfrak{Pr}}}_{R-\Mod}}^{-1}(\mathcal{L})$, then for every $M \in \#\mathcal{L}$, $N \in {{\underline {\mathfrak{Pr}}}_{R-\Mod}}^{-1}(M_0)$. Then,  for every $M  \in \#\mathcal{L}$, $\overline{N} \in {{\underline {\mathfrak{Pr}}}_{\C(R)}}^{-1}(M)$ (see  \cite[Lemma 3.7]{Subcom}). Therefore, $\overline{N} \in {{\underline {\mathfrak{Pr}}}_{\C(R)}}^{-1}(\#\mathcal{L})=\mathcal{\widetilde{G}}$, hence $N\in \mathcal{G}$. Conversely, let $N \in \mathcal{G}$ and $M \in \mathcal{L}$ then, $\overline{N} \in \mathcal{\widetilde{G}}$ and $\overline{M} \in \#\mathcal{L}$. Then,  $\overline{N} \in {{\underline {\mathfrak{Pr}}}_{\C(R)}}^{-1}(\overline{M})$ by assumption. Then,  $N \in {{\underline {\mathfrak{Pr}}}_{R-\Mod}}^{-1}(M)$ by  \cite[Lemma 3.7]{Subcom}. Then,  $N \in {{\underline {\mathfrak{Pr}}}_{R-\Mod}}^{-1}(\mathcal{L})$.
\end{proof}

Recall that an object $M$ of an abelian category $\mathscr A$ is said to be Gorenstein projective, if  there exists an exact and $\Hom_R(-, \mathcal{P}_\A)$-exact complex of projective objects $$\cdots \to P_{-1} \to P_0 \to P_{1} \to \cdots$$ such that $M =\Ker(P_0 \to P_1)$ (see \cite[Definition 10.2.1]{Enochs}). We use $\mathcal{GP_\A}$ to denote the class of all Gorenstein projective objects of $\A$. The authors in \cite{Liu}, proved that the class $\mathcal{GP}_{\C(R)}^\perp$ is inside the class of exact complexes with cycles in $\mathcal{GP}_{R-\Mod}^\perp$, that is the class $\mathcal{\widetilde{GP_{R-\Mod}^\perp}}$  (see \cite[Proposition 3.4 and Remark 3.3]{Liu}). Here, we show  when these two classes coincide.

\begin{prop}\label{pro-GP}
	For any ring,  $\mathcal{GP}_{\C(R)}^\perp =\mathcal{\widetilde{GP_{R-\Mod}^\perp}}$ if and only if $\Hom_{\K(R)}(M, N)=0$ for any Gorenstein projective complex $M$ and any complex $N \in \mathcal{\widetilde{GP_{R-\Mod}^\perp}}$.
\end{prop}
\begin{proof}
	We have ${{\underline {\mathfrak{Pr}}}_{R-\Mod}}^{-1}(\mathcal{GP_{R-\Mod}})= \mathcal{GP}_{R-\Mod}^\perp$ by \cite[Corollary 2.28]{SubprojAb}, then, by Theorem \ref{thm-5-2}, ${{\underline {\mathfrak{Pr}}}_{\C(R)}}^{-1}(\#\mathcal{GP_{R-\Mod}})= \mathcal{\widetilde{GP_{R-\Mod}^\perp}}$ if and only if $\Hom_{\K(R)}(M, N)=0$ for any $M \in \#\mathcal{{GP_{R-\Mod}}}$ and $N \in \mathcal{\widetilde{GP_{R-\Mod}^\perp}}$. But $\#\mathcal{GP_{R-\Mod}}= \mathcal{GP}_{\C(R)}$, by \cite[Theorem 2.2]{Yang}. Then,  ${{\underline {\mathfrak{Pr}}}_{\C(R)}}^{-1}(\#\mathcal{GP_{R-\Mod}})={{\underline {\mathfrak{Pr}}}_{\C(R)}}^{-1}(\mathcal{GP_{\C(R)}})= \mathcal{GP}_{\C(R)}^\perp$ (see \cite[Corollary 2.28]{SubprojAb}). Therefore, $\mathcal{GP}_{\C(R)}^\perp= \mathcal{\widetilde{GP_{R-\Mod}^\perp}}$ if and only if $\Hom_{\K(R)}(M, N)=0$ for any $M$ Gorenstein projective complex and any $N \in \mathcal{\widetilde{GP_{R-\Mod}^\perp}}$. 
\end{proof}

Now, we turn our attention to the third question.  In the study of this question, a new    type of classes appears naturally which are defined as follows:    First, we fix some notations and recall some facts: For complexes $X$ and $Y$, we let $\Hom^{\bullet}(X,Y)$ denote the complex of abelian groups with $$ \Hom^{\bullet}(X,Y)_{n}=\prod_{i\in \mathbb{ Z}}\Hom_R(X_{i},Y_{i+n})$$ and $$ d_{n}^{\Hom^{\bullet}(X,Y)}(\psi)=(d^Y_{i+n}\psi_i -(-1)^{n}\psi_{i-1}d_i^X)_{i\in \mathbb{ Z}}.$$ 
Note that for every $n \in \mathbb{Z}$, $$Z_n(\Hom^{\bullet}(X,Y))=\Hom_{\C(R)}(X[n],Y)=\Hom_{\C(R)}(X,Y[-n])$$ and $$H_n(\Hom^{\bullet}(X,Y))=\Hom_{\K(R)}(X[n],Y)=\Hom_{\K(R)}(X,Y[-n]).$$ For every complex $X$, $\Hom^{\bullet}(X,-)$ is a left exact functor from the category of complexes of modules to the category of complexes of abelian groups.

\begin{Def}
	Given a class of modules $\mathcal{L}$, a complex $X$ is said to be a $dg \mathcal{L}$ complex, if $X_n \in \mathcal{L}$, for each $n \in \mathbb{Z}$, and $\Hom^{\bullet}(X,G)$ is exact whenever $G$ is an exact complex with cycles in ${{\underline {\mathfrak{Pr}}}_{R-\Mod}}^{-1}(\mathcal{L})$. We denote the class of $dg \mathcal{L}$ complexes by  $dg\mathcal{\widetilde{L}}$.
\end{Def}

The terminology used in this definition is inspired from   $dg \mathcal{L}$ complexes introduced by Gillespie \cite{Gillespie}   based on the fact that,  when $(L,G)$ is the classical cotorsion pair $(Proj, R-Mod)$,  $dg\mathcal{\widetilde{L}}$  is nothing but the classical DG-projective complexes.

\begin{lem}\label{lem-dgL}
	If $\mathcal{L}$ is a class which contains $0$, then $\C^-(\mathcal{L})\subseteq dg\mathcal{\widetilde{L}}$.
\end{lem}
\begin{proof}
	Set $\mathcal{G}={{\underline {\mathfrak{Pr}}}_{R-\Mod}}^{-1}(\mathcal{L})$ and let $f:X \to N$ be a morphism of complexes with $N \in \mathcal{\widetilde{G}}$. $N$ is exact and $Z_n(N) \in {{\underline {\mathfrak{Pr}}}_{R-\Mod}}^{-1}(\mathcal{L})$, for each $n \in \mathbb{Z}$, hence $N$ is $\Hom_R(Proj, -)$-exact and every morphism $X_n \to Z_n(N)$ factors through a projective module. Therefore $f:X \to N$ is null-homotopic, by Lemma \ref{lem-nul2}. Thus, $\Hom^{\bullet}(X,N)$ is exact whenever $N$ is a $\mathcal{\widetilde{G}}$ complex.
\end{proof}

Now, our third main result of this section is given as follows.
\begin{thm} \label{thm-5-3}
	Let $\mathcal{L}$ and $\mathcal{G}$ be two classes of modules such that  $0,R \in \mathcal{L}$. Then,  the following conditions are equivalent.
	\begin{enumerate}
		\item  ${{\underline {\mathfrak{Pr}}}_{R-\Mod}}^{-1}(\mathcal{L})=\mathcal{G}$.
		\item $ {{\underline {\mathfrak{Pr}}}_{\C(R)}}^{-1}(\C^b(\mathcal{L}))=\mathcal{\widetilde{G}}.$
		\item $ {{\underline {\mathfrak{Pr}}}_{\C(R)}}^{-1}(\C^-(\mathcal{L}))=\mathcal{\widetilde{G}}.$
		\item ${{\underline {\mathfrak{Pr}}}_{\C(R)}}^{-1}(dg\mathcal{\widetilde{L}})=\mathcal{\widetilde{G}}.$
	\end{enumerate}	
\end{thm}

\begin{proof}
	$1. \Rightarrow 4.$ Let us prove that $ {{\underline {\mathfrak{Pr}}}_{\C(R)}}^{-1}(dg\mathcal{\widetilde{L}}) \subseteq \mathcal{\widetilde{G}}$. For let  $N \in {{\underline {\mathfrak{Pr}}}_{\C(R)}}^{-1}(dg\mathcal{\widetilde{L}})$, then  $N \in {{\underline {\mathfrak{Pr}}}_{\C(R)}}^{-1}(\underline{R}[n])$ for every $n \in \mathbb{Z}$, by Lemma \ref{lem-dgL}, so $N$ is exact by \cite[Corollary 3.17]{Subcom}. Now, let  $n \in \mathbb{Z}$ and $L \in \mathcal{L}$, then  $N \in {{\underline {\mathfrak{Pr}}}_{\C(R)}}^{-1}(\underline{L}[n])$ for every $n \in \mathbb{Z}$, by Lemma \ref{lem-dgL}. Then, by \cite[Lemma 3.6]{Subcom}, $ Z_n(N) \in {{\underline {\mathfrak{Pr}}}_{R-\Mod}}^{-1}(L)$ for every  $L \in \mathcal{L}$. Then,  $ Z_n(N) \in {{\underline {\mathfrak{Pr}}}_{R-\Mod}}^{-1}(\mathcal{L})= \mathcal{G}$ for every $n \in \mathbb{Z}$. Thus, $ N \in \mathcal{\widetilde{G}}$. Now, let us prove that $  \mathcal{\widetilde{G}} \subseteq {{\underline {\mathfrak{Pr}}}_{\C(R)}}^{-1}(dg\mathcal{\widetilde{L}})$. Let $N \in \mathcal{\widetilde{G}}$ and $M \in dg\mathcal{\widetilde{L}}$, then $\Hom_{\K(R)}(M,N)=0$ since $\Hom^\bullet(X,N)$ is exact (see \cite[Lemma 2.1]{Gillespie}). For every $n \in \mathbb{Z}$, $Z_n(N) \in \mathcal{G}= {{\underline {\mathfrak{Pr}}}_{R-\Mod}}^{-1}(\mathcal{L})$, hence $N_n \in {{\underline {\mathfrak{Pr}}}_{R-\Mod}}^{-1}(\mathcal{L})$ since ${{\underline {\mathfrak{Pr}}}_{R-\Mod}}^{-1}(\mathcal{L})$ is closed under extensions. Then,  $N \in {{\underline {\mathfrak{Pr}}}_{\C(R)}}^{-1}(dg\mathcal{\widetilde{L}})$ by \cite[Theorem 3.11]{Subcom}. Thus, $N \in {{\underline {\mathfrak{Pr}}}_{\C(R)}}^{-1}(dg\mathcal{\widetilde{L}})$.\\	
	$4. \Rightarrow 1.$ Let $N \in {{\underline {\mathfrak{Pr}}}_{R-\Mod}}^{-1}(\mathcal{L})$, then for every $M \in dg\mathcal{\widetilde{L}}$, $N \in {{\underline {\mathfrak{Pr}}}_{R-\Mod}}^{-1}(M_0)$ since $M_0 \in \mathcal{L}$. Then,  for every $M  \in dg\mathcal{\widetilde{L}}$, $\overline{N} \in {{\underline {\mathfrak{Pr}}}_{\C(R)}}^{-1}(M)$ (see  \cite[Lemma 3.7]{Subcom}). Therefore, $\overline{N} \in {{\underline {\mathfrak{Pr}}}_{\C(R)}}^{-1}(dg\mathcal{\widetilde{L}})=\mathcal{\widetilde{G}}$, hence $N\in \mathcal{G}$. Conversely, let $N \in \mathcal{G}$ and $M \in \mathcal{L}$ then, $\overline{N} \in \mathcal{\widetilde{L}}$ and $\overline{M} \in dg\mathcal{\widetilde{L}}$ (see Lemma \ref{lem-dgL}). Then,  $\overline{N} \in {{\underline {\mathfrak{Pr}}}_{\C(R)}}^{-1}(\overline{M})$ by assumption. Then,  $N \in {{\underline {\mathfrak{Pr}}}_{R-\Mod}}^{-1}(M)$ by  \cite[Lemma 3.7]{Subcom}. Then,  $N \in {{\underline {\mathfrak{Pr}}}_{R-\Mod}}^{-1}(\mathcal{L})$ \\
	$(4.\Leftrightarrow 1.)\Rightarrow 2.$ It is clear that $\mathcal{\widetilde{G}} \subseteq {{\underline {\mathfrak{Pr}}}_{\C(R)}}^{-1}(\C^b(\mathcal{L}))$ since $\C^-(\mathcal{L}) \subseteq dg\mathcal{\widetilde{L}}$ by Lemma \ref{lem-dgL}. Conversely, let  $N \in {{\underline {\mathfrak{Pr}}}_{\C(R)}}^{-1}(\C^-(\mathcal{L}))$ then  $N \in {{\underline {\mathfrak{Pr}}}_{\C(R)}}^{-1}(\underline{R}[n])$ for every $n \in \mathbb{Z}$, so $N$ is exact by \cite[Corollary 3.17]{Subcom}. Now, let  $n \in \mathbb{Z}$ and $L \in \mathcal{L}$, then  $N \in {{\underline {\mathfrak{Pr}}}_{\C(R)}}^{-1}(\underline{L}[n])$ for every $n \in \mathbb{Z}$. Then, by \cite[Lemma 3.6]{Subcom}, $ Z_n(N) \in {{\underline {\mathfrak{Pr}}}_{R-\Mod}}^{-1}(L)$ for every  $L \in \mathcal{L}$. Then,  $ Z_n(N) \in {{\underline {\mathfrak{Pr}}}_{R-\Mod}}^{-1}(\mathcal{L})= \mathcal{G}$ for every $n \in \mathbb{Z}$. Thus, $ N \in \mathcal{\widetilde{G}}$.\\
	For $2. \Rightarrow 1.$ and $3. \Rightarrow 1.$ we use the same arguments of $4. \Rightarrow 1.$\\
	$(2.\Leftrightarrow 4.)\Rightarrow 3.$ is clear since $\C^b(\mathcal{L}) \subseteq \C^-(\mathcal{L}) \subseteq dg\mathcal{\widetilde{L}}$ (by Lemma \ref{lem-dgL}).\\
\end{proof}

Apllying Theorem \ref{thm-5-3} to the class of finitely presented modules,  we get the following  characterization of flat complexes.

\begin{cor}\label{cor-l1}  
	The following conditions are equivalent for a complex $F$.
	\begin{enumerate}
		\item $F$ is exact and every cycle is a flat module.
		\item Every morphism $X \to F$, with $X$ is a finitely presented complex, factors through a projective complex.
		\item Every morphism $X \to F$, with $X$ is a bounded below complex of finitely presented modules, factors through a projective complex.
		\item Every morphism $X \to F$, with $X$ is a $dg\mathcal{\widetilde{FP}}$ complex, factors through a projective complex.
	\end{enumerate}
\end{cor}

Apllying Theorem \ref{thm-5-3}, to  $ {{\underline {\mathfrak{Pr}}}_{R-\Mod}}^{-1}(Proj)= R-\Mod$,  we get the following characterization of exact complexes.

\begin{cor}\label{cor-l2}
	The following conditions are equivalent for a complex $F$.
	\begin{enumerate}
		\item $F$ is exact.
		\item Every morphism $X \to F$, with $X$ is a bounded complex of projective modules, factors through a projective complex.
		\item Every morphism $X \to F$, with $X$ is a bounded complex below of projective modules, factors through a projective complex.
		\item Every morphism $X \to F$,with $X$ is a $DG$-projective complex, factors through a projective complex.
	\end{enumerate}
\end{cor}

Finally, applying Theorem \ref{thm-5-3} to  $ {{\underline {\mathfrak{Pr}}}_{R-\Mod}}^{-1}(R-\Mod)=Proj$,  we get the following characterization of projective complexes.

\begin{cor}\label{cor-l3}
	The following conditions are equivalent for a complex $F$.
	\begin{enumerate}
		\item $F$ is exact and every cycle is a projective module.
		\item Every morphism $X \to F$, with $X$ is  bounded, factors through a projective complex.
		\item Every morphism $X \to F$, with $X$ is bounded below, factors through a projective complex.
		\item Every morphism $X \to F$, factors through a projective complex.
	\end{enumerate}
\end{cor}

Driss Bennis:   CeReMAR Center; Faculty of Sciences, Mohammed V University in Rabat, Rabat, Morocco.

\noindent e-mail address: driss.bennis@um5.ac.ma; driss$\_$bennis@hotmail.com

J. R. Garc\'{\i}a Rozas: Departamento de  Matem\'{a}ticas,
Universidad de Almer\'{i}a, 04071 Almer\'{i}a, Spain.

\noindent e-mail address: jrgrozas@ual.es

Hanane Ouberka:   CeReMAR Center; Faculty of Sciences, Mohammed V University in Rabat, Rabat, Morocco.

\noindent e-mail address: hanane$\_$ouberka@um5.ac.ma; ho514@inlumine.ual.es 

Luis Oyonarte: Departamento de  Matem\'{a}ticas, Universidad de
Almer\'{i}a, 04071 Almer\'{i}a, Spain.

\noindent e-mail address: oyonarte@ual.es

\end{document}